\documentclass{amsart}
\usepackage[mathscr]{eucal}
\usepackage{amssymb, amsmath,array, amscd}
\usepackage{url}

\usepackage[active]{srcltx}

\input xy
\xyoption{all}

\theoremstyle{change}
\newtheorem{thm}{Theorem}[section]

\newtheorem*{ThmA}{Theorem~A}
\newtheorem*{CorB}{Corollary~B}
\newtheorem*{ThmC}{Theorem~C}
\newtheorem{prop}[thm]{Proposition}

\newtheorem{lemma}[thm]{Lemma}
\newtheorem{definition}[thm]{Definition}
\newtheorem{remark}[thm]{Remark}

\numberwithin{equation}{section}

\newcommand{\cF}{\mathcal{F}}
\newcommand{\cO}{\mathcal{O}}
\newcommand{\cG}{\mathcal{G}}
\newcommand{\cC}{\mathcal{C}}
\newcommand{\cU}{\mathcal{U}}

\newcommand{\Om}{\Omega}
\newcommand{\om}{\omega}
\newcommand{\la}{\lambda}

\newcommand{\C}{\mathbb{C}}
\newcommand{\HH}{\mathbb{H}}
\newcommand{\N}{\mathbb{N}}
\newcommand{\Q}{\mathbb{Q}}
\newcommand{\R}{\mathbb{R}}
\newcommand{\Z}{\mathbb{Z}}
\newcommand{\PP}{\mathbb{P}}

\newcommand{\f}{\phi}

\newcommand{\sing}{\mathrm{Sing}}
\renewcommand{\=}{:=}
\renewcommand{\div}{\mathrm{div}}
\newcommand{\ord}{\mathrm{ord}}
\newcommand{\dto}{\dashrightarrow}
\newcommand{\ind}{\mathrm{Ind}}
\newcommand{\NS}{\mathrm{NS}_{\mathbb Q}}
\newcommand{\kod}{\mathrm{kod}}
\newcommand{\aut}{\mathrm{Aut}}
\newcommand{\PSL}{\mathrm{PSL}}
\newcommand{\GL}{\mathrm{GL}}

\begin{document}
\title{Foliations invariant by rational maps}
\author{Charles Favre}
\address{Charles Favre \\
Centre de Math\'ematiques Laurent Schwartz \\
\'Ecole Polytechnique \\
91128 Palaiseau Cedex France
.}
\email{favre@math.polytechnique.fr}

\author{Jorge Vit\'orio Pereira}
 \address{Jorge Vit\'orio Pereira \\ IMPA, Estrada
 Dona  Castorina, 110\\
 22460-320, Rio de Janeiro, RJ, Brazil}
 \email{jvp@impa.br}

\thanks{The first author was partially supported by the project ECOS-Sud No.C07E01, and the second  by CNPq-Brazil.}
 \date{\today}

\begin{abstract}
We give a classification of pairs $(\cF, \f)$ where $\cF$ is a  holomorphic foliation on a projective surface and $\f$ is a non-invertible dominant rational map preserving $\cF$.
\end{abstract}

\maketitle


\section{Introduction}



Our goal is to give a classification of pairs $(\cF, \f)$ where $\cF$ is a holomorphic singular foliation on a projective surface $X$, and $\f: X \dto X$ is a dominant rational map preserving $\cF$.
When $\f$ is birational, the situation is completely understood, see~\cite{CF,JVPS}.
When $\f$ is not invertible, to the best of our knowledge the only result available in the literature is due to M. Dabija and M. Jonsson~\cite{DJ}, and concerns the classification of pencil of curves in $\PP^2$ preserved by holomorphic self-maps of the projective plane.

Our approach to this  problem is based on foliated Mori theory~\cite{Br,Br2,McQ,mendes} which address the problem of classification of holomorphic foliations on surfaces by describing the numerical properties of their cotangent bundle in suitable birational models. Our main result gives a classification of pairs $(\cF, \f)$ as above up to birational conjugacy.
Recall that a foliation $\cF$ has a rational first integral if it is defined by a pencil of curves in which case one says $\cF$ is tangent to this pencil.
\begin{ThmA}
 Suppose $\cF$ is a holomorphic singular foliation on a projective surface $X$ without rational first
integral, and $\f: X \dto X$ is a  dominant non-invertible rational map preserving $\cF$.
Then up to a birational conjugacy, there exists a $\f$-invariant Zariski open dense subset $\cU \subset X$
such that $\cU$ is a quotient of $\C^2$ by a discrete subgroup $\Gamma \subset \mathrm{Aff}(\C^2)$ acting discontinuously on $\C^2$, $X$ is an equivariant compactification of $\C^2/\Gamma$ and:
\begin{itemize}
 \item $\f$ lifts to an affine map on $\C^2$;
 \item $\cF$ is defined in $\C^2$ either by $dx$ or by $d(y+ e^x)$ in suitable affine coordinates.
\end{itemize}
\end{ThmA}
In the case $\f$ is  birational, a similar statement holds whenever the degrees of $\f$ grow exponentially, see~\cite[Corollary 1.3]{CF}.

Note that this implies $\cF$ is non singular (in the orbifold sense)  on the open set $\cU$, and has a Liouvillian first integral as in the case of
foliations invariant by birational maps, see~\cite[Corollary 7.4]{CF}. We shall call any rational map $\f$  which lifts to an affine map  in $\C^2$ like in the theorem \emph{Latt\`es-like}. This class already appears in the classification of endomorphisms of $\PP^k$, $k\ge 1$ having a non-trivial centralizer by Dinh and Sibony~\cite{Dinh-Sibony}.

Let us comment briefly on which surfaces $\cU$ may appear.
When $\Gamma$ is a subgroup of $(\C^2,+)$ of rank $r \in \{ 1,2,3,4\}$, then $\cU$ is either
an abelian surface ($r=4$); a locally trivial fibration over $\C^*$ with fiber some elliptic curve ($r=
3$); $\C\times E$ for some elliptic curve, or $\C^* \times \C^*$ ($r=2$); or $\C\times \C^*$ ($r=1$). In particular any projective\footnote{Otherwise more examples arise like non-k\"ahlerian Kodaira surfaces.}  compactification of $\cU$ is birational to a torus, to $\PP^1 \times E$ for some elliptic curve $E$, or to $\PP^1 \times \PP^1$.  Note that each of these surfaces admit Latt\`es-like maps.
In general, when $\Gamma$ is no longer a subgroup of  $(\C^2,+)$, one obtains finite quotients of these examples.

In fact, there are strong restrictions for a Latt\`es-like map on a surface $X$ to preserve a foliation in $\C^2$ which induces a holomorphic singular $\f$-invariant foliation on $X$. Our results are actually more precise, and we give normal forms for both the foliation and the map, see Theorems~\ref{T:clas-kod0},~\ref{T:clas-kod1} below in the case $\cF$ has no first integral and Theorem~\ref{T:clas-fibr} when $\cF$ is tangent to a pencil of curves.
As a consequence of the classification, we obtain the following result which extends~\cite[Corollaire~1.3]{CF} to non-invertible maps.
\begin{CorB}
Suppose $\f: X \dto X$ is a  dominant non-invertible rational surface map which preserves a holomorphic foliation $\cF$ which is not tangent to a rational or an elliptic fibration. Then $\f$ preserves at least two foliations.
\end{CorB}
The assumption is necessary, since the map $(x,y) \mapsto (x^2, x y^2)$ on $\C^2$ does not preserve any other foliation than $\{x = \mathrm{cst}\}$.

Our approach also gives an alternative to the delicate analysis of the reduced fibers done by Dabija-Jonsson in~\cite{DJ}. It allows us to extend their results to arbitrary foliations. Note that the following result is not an immediate consequence of the previous results since it gives a classification up to $\mathrm{PGL}(3,\C)$. Recall that the degree of a foliation $\cF$ in $\PP^2$ is the number of tangencies between $\cF$ and a generic line of $\PP^2$.
\begin{ThmC}
Suppose $\f: \PP^2 \to \PP^2$ is a holomorphic map of degree $d\ge 2$ preserving a foliation $\cF$. Then
$\deg(\cF) = 0$ or $1$, and in appropriate homogeneous coordinates $[x:y:z]$ on $\PP^2$, one of the following cases holds:
\begin{enumerate}
 \item $\cF$ is the pencil of lines given by $d(x/y)$, and $\f= [P(x,y):Q(x,y):R(x,y,t)]$ with $P,Q,R$
homogeneous polynomials of degree $d$;
 \item $\cF$ is induced by $\, d \log ( x^{\la} y z^{-1-\la} ) $ with $\la \in \C \setminus \Q$, and $\f=
 [x^d:y^d:z^d]$;
 \item $\cF$ is induced by $\, d \log ( x^{\xi} y z^{-1-\xi} ) $ with $\xi$ a primitive $3$-rd root of unity
and $\f = [z^d:x^d:y^d]$;
 \item $\cF$ is induced by $\, d (x^p y^q / z^{p+q} )$  with $p,q \in \N^*$,
$p\neq q$, $\gcd\{ p,q \} =1$, and $\f =[x^d:y^d:R(x,y,z)]$ with $R = z^\delta \prod_{i=1}^l (z^{p+q} + c_i x^py^q)$ with $d = \delta + l(p+q)$, $c_i \in \C^*$.
 \item $\cF$ is induced by the $1$-form $d\log (xyz^{-2})$, $\f =[y^d:x^d:R(x,y,z)]$ with $R = z^\delta
\prod_{i=1}^l (z^2 + c_i xy)$ and $d = \delta + 2l$, $c_i \in \C^*$.
\end{enumerate}
\end{ThmC}
In~\cite{FP}, we extend our classification to rational  maps preserving webs, inspired by the work of Dabija-Jonsson~\cite{DJ2} on endomorphisms preserving families
of lines.

\medskip

The plan of the paper is as follows. In Section \ref{S:reduced}, we recall basic facts about foliations and their singularities. We then describe how the cotangent bundle of a foliation behaves under the action of a rational map  in Proposition~\ref{P:var}. From this key computation, we deduce a simple proof of Theorem~C.
In Section~\ref{S:mori}, we present the basics of foliated Mori theory, following~\cite{Br}. We deduce from it two general statements (Proposition~\ref{P:key2} and~\ref{P:key3}) that are important intermediate results.
Section~\ref{S:clas} contains the classification up to finite covering  and birational conjugacy of invariant foliations without rational first integral of Kodaira dimension $0$ and $1$ (Theorem~\ref{T:clas-kod0} and~\ref{T:clas-kod1}), and of invariant fibrations (Theorem~\ref{T:clas-fibr}). The proofs of Theorem~A and Corollary~B are then given at the end of the paper.

\medskip
\noindent
{\bf Acknowledgements.} We thank Serge Cantat, Romain Dujardin and Mattias Jonsson for their comments.

\section{Transformation of the cotangent bundle}\label{S:reduced}
In this section by a surface $X$ we  mean a compact complex surface with at most quotient singularities.

\subsection{Foliations on smooth surfaces}
Let us consider first a smooth surface $X$. We denote by $TX$ its tangent bundle, by $\Om^1_X$ its sheaf of holomorphic $1$-forms and by $K_X \= \Om^2_X$ its canonical line bundle.

A (singular holomorphic) foliation $\cF$ on a surface $X$ is determined by a section with isolated zeroes $v \in H^0(X, TX \otimes T^*\cF)$
for some line bundle $T^*\cF$ on $X$ called the cotangent bundle of $\cF$. Two sections $v, v'$ define the same foliation if and only if  they differ by a global nowhere vanishing holomorphic function. The zero locus of the section $v$ is a finite set $\sing (\cF)$ called the singular locus of $\cF$. A point $x \notin \sing(\cF)$ is said to be regular.

Concretely, given an open contractible Stein cover $\{ U_i\}$ of $X$, $\cF$ is described in each $U_i$ by some vector field $v_i$ with isolated zeroes such that $v_i = g_{ij} v_j$ on $U_i \cap U_j$ for some non-vanishing holomorphic functions $g_{ij} \in \cO^*(U_i \cap U_j)$. The cocycle $\{ g_{ij}\}$ determines the cotangent bundle of $\cF$, and the collection $\{v_i\}$ induces a global section $v \in H^0(X,TX \otimes T^*\cF)$.

Integral curves of the vector fields $\{ v_i \}$ patch together to form the leaves of $\cF$.
Outside $\sing(\cF)$, a local section of $T^*\cF$ is given by a holomorphic $1$-form along the leaves.

A foliation $\cF$ can be described in a dual way by a section with isolated singularities  $\om \in H^0(X, \Om^1_X \otimes N\cF)$
for some line bundle $N\cF$ called the normal bundle of $\cF$. In a  contractible Stein open cover,
$\om$ is determined by holomorphic $1$-forms $\om_i$. The tangent spaces of leaves are then
 the kernels of $1$-forms $\om_i$.  Notice  that the collection of $1$-forms $\{ \om_i\}$ determine the same foliation as the collection of vector fields $\{ v_i \}$ if and only if  $\om_i(v_i) \equiv 0$ for all $i$.

At a regular point of $\cF$, one can contract a local section of $T\cF$ with a holomorphic $2$-form
yielding a local section of $N^*\cF$. We thus have the isomorphism:
\begin{equation}\label{E:canonico}
K_X = N^* \cF \otimes T^* \cF
\end{equation}

Suppose $\f : Y \to X$ is a dominant holomorphic map between two smooth surfaces $X, Y$, and $\cF$ is a
foliation on $X$ determined by a collection of $1$-forms $\{\om_i\}$
on some open  cover $\{U_i\}$ of $X$. The holomorphic $1$-form $\f^* \om_i$ on $\f^{-1} (U_i)$ may have  in general non-isolated zeroes.
We denote by $\hat{\om}_i$ any holomorphic $1$-form on $\f^{-1} (U_i)$ with isolated zeroes which is proportional to $\f^* \om_i$.
The collection $\{ \hat{\om}_i\}$ then defines a holomorphic foliation that we call the pull-back of $\cF$ by $\f$ and denote by $\f^* \cF$.

\smallskip

A singular point $x$ for a foliation $\cF$ is said to be reduced if the foliation is locally defined by a vector field $v$ whose linear part is not nilpotent, and such that the quotient of its two eigenvalues is not a positive rational number (it may be zero or infinite). Reduced singularities satisfy the following property: for any composition of point blow-ups $\pi$ above $p$,
the invertible sheaf $T^*(\pi^* \cF) - \pi^* T^*\cF$ is determined by an \emph{effective} $\pi$-exceptional divisor. Note that in the smooth case then this difference has full support on the set of $\pi$-exceptional divisors.

A separatrix at $x$ is a germ of curve $C$ passing through $x$ such that $C\setminus \{ x\}$ is a leaf of $\cF$. The singularity of $\cF$ at $x$ is said to be dicritical if there exist infinitely many separatrices at this point. A reduced singularity is not dicritical.

It is a theorem of Seidenberg that for any foliation $\cF$ there exists a composition of point blow-ups $\pi: Y \to X$ such that $\pi^* \cF$ is reduced.
A reduced foliation on a surface is a foliation with all its singularities reduced. If $\cF$ is a reduced foliation on $X$, and $\pi : Y \to X$ is  a composition of point blow-ups then $\pi^* \cF$ is again reduced.

\subsection{Foliations on  surfaces with quotient singularities}

We shall also work with surfaces with quotient singularities. A foliation on such a surface is
a foliation on its smooth part.
Suppose $(X,p)$ is a quotient singularity locally isomorphic to $\C^2/G$ where $G$ is a finite subgroup of $\mathrm{GL}(2,\C)$. A foliation $\cF$ on $(X,p)$ is the image of a foliation on $\C^2$ which is $G$-invariant.
A foliation on $(X,p)$ is said to have a reduced singularity (resp. to be smooth) at $p$ iff its lift to $\C^2$ has reduced singularities (resp. is smooth).

On a surface $X$ with at most quotient singularities, there exists a natural definition of intersection of divisors. Note that the intersection product of two curves needs not be an integer but is instead a rational number.  One then defines the (rational) Neron-Severi group of  $X$, $\NS(X)$, as  the group of $\mathbb Q$-divisors of $X$ modulo numerical equivalence.

Since $X$ is singular, $T^*\cF$ is not necessarily a line-bundle. It is only a torsion-free sheaf locally free outside the singular set of $X$.
However a suitable power of $T^* \cF$ is  a line-bundle, and consequently determines a class in $\NS(X)$ which will also be denoted by $T^* \cF$.
For any birational morphims $\pi: X' \to (X,p)$, one can then define
$\pi^* T^*\cF \in \NS(X')$. If $\cF$ has a reduced singularity at $p$ and $X'$ has only
quotient singularities, then $T^*(\pi^*\cF) - \pi^* (T^*\cF)$ is determined by an effective $\pi$-exceptional divisor just as in the smooth case (note however that the coefficients of this divisor may be non-integral rational numbers). Again this difference has full support on the set of $\pi$-exceptional divisors if $\cF$ is smooth at $p$.

\subsection{Rational maps}

Suppose $\pi : X' \to X$ is a birational morphism and $\cF'$ is a foliation on $X'$.
Then we can push-forward $\cF'$ outside the exceptional components of $\pi$. This defines a foliation
on the complement of a finite set in $X$. By Hartog's theorem, this foliation extends to $X$ in a unique way: we denote it by $\pi_* \cF'$.

If $\f: Y \dto X$ is a dominant meromorphic map and $\cF$ is a foliation on $X$ then $\f^* \cF$ is defined as follows.
Let $\Gamma$ be a desingularization of the graph of $\f$. Thus there are two holomorphic maps
$\pi: \Gamma \to Y$ and $f: \Gamma \to X$ such that $\pi$ is a composition of point blow-ups and $f = \f \circ \pi$ as
as indicated in the diagram below. By definition, $\f^* \cF$ is the foliation $\pi_* f^* \cF$ on $X$.
\[
 \xymatrix{
 &\Gamma \ar[dl]_{\pi} \ar[dr]^f &
  \\
 X  \ar@{-->}[rr]^{\f} && Y}
\]
In the sequel, a foliated surface $(X,\cF)$ is a compact complex surface with at most quotient singularities endowed with a singular holomorphic foliation. A map $\phi : (Y,\cG) \dto (X,\cF)$ of foliated surfaces is a dominant meromorphic map $\f : Y \dto X$ such that  $ \cG = \f^* \cF$.

\smallskip

Let us introduce a little more notation.
We let $\ind(\f)$ be the set of indeterminacy points of $\f$. In terms of the diagram above, it is the finite set of points $p \in X$ such that $f \pi^{-1} (p)$ has positive dimension.
For any Cartier $\Q$-divisor $D$ in $X$, we set $\f^* D \= \pi_* f^* D$. It is a linear map which preserves effectivity.

Then $\f$ induces maps between the Neron-Severi groups of $X$ and $Y$ that we again denote by  $\f^* : \NS(X) \to \NS(Y)$.

\subsection{Pull-back of foliations}

Denote by $\Delta_\f$ the divisor determined by the vanishing of the Jacobian determinant of $\f$. It is locally defined by the vanishing of $\det D\f$, where $D\f$ denotes the differential of $\f$. Its support is the critical set of $\f$, and it satisfies the equation
\begin{equation}\label{E:jacob}
 K_X = \f^* K_Y + \Delta_\f~.
\end{equation}

The following result is our key technical tool.
When $E$ is an irreducible curve, and $Z$ is any divisor we denote by $\ord_E(Z)\in \Z$ the order of vanishing of $Z$ at a generic point of $E$.

\begin{prop}\label{P:var}
Suppose $\f: (Y, \cG) \dto (X,\cF) $ is a dominant meromorphic map
between foliated surfaces with at most quotient singularities.  Then one has  $\f^* T^*\cF =  T^*\cG - D$ in $\NS(Y)$
for some (non necessarily effective) divisor with  support  included in the critical set of $\f$ and satisfying
$D \le \Delta_\f$.

Pick any critical component $E$, and when it is contracted to a point assume that
$\f(E)$ is a reduced singularity of $\cF$.
\begin{enumerate}
\item If $E$ is generically transverse to $\cG$ then $\ord_E(D) = \ord_E(\Delta_\f)>0$;
\item If $E$ is $\cG$-invariant, then $\ord_E(D) \ge 0$.
\end{enumerate}
In particular, if $\cF$ has only reduced singularities then $D$ is effective.
\end{prop}

\begin{proof}[Proof of Proposition~\ref{P:var}]
Suppose $\cF$ is given by a collection of holomorphic $1$-forms $\{\om_i\}$ on
an open Stein cover $\{U_i\}$ of $X\setminus \mathrm{Sing}(X)$. On the open set $\f^{-1}(U_i)\setminus \ind(\f)\cup \mathrm{Sing}(Y)$, we may write
$\f^* \om_i = h_i \cdot \hat{\om}_i$ with $\hat{\om}_i$ a holomorphic $1$-form
with isolated zeroes and $h_i$ a holomorphic function whose zero set is included in the critical set of $\f$.
The divisors $\div(h_i)$ patch together and yields a global effective divisor $D_0$.
Outside $\sing(\cF)\cup \mathrm{Sing}(X)$, $\om_i$ is a local generator for the invertible sheaf $N^*\cF$ over $U_i$, and the same is true with $\hat{\om}_i$ for $N^* \cG$ over $\f^{-1}(U_i) \setminus \sing (\cG)\cup \mathrm{Sing}(Y)$.
Whence $\f^* N^*\cF = N^* \cG -D_0$. By~\eqref{E:canonico} and~\eqref{E:jacob}, we get $\f^* T^*\cF = T^*\cG + D_0 - \Delta_\f$. This proves the first part of the proposition with $D = \Delta_\f -D_0$.

For the second part, we pick an irreducible component $E$ of the critical set of $\f$.
By our assumption we know that either $E$ is not contracted, or it
is contracted to a reduced singularity of $\cF$.

Suppose first $E$ is transversal to $\cG$. Then it cannot be contracted to a point since a reduced singularity admits only finitely many separatrices. At a generic point on $E$, and in suitable coordinates we may write $\f(x_1,x_2) = (x_1^a,x_2)$. On the other hand, we may suppose that $\cF$ is determined by the form $dy_2$, hence $\f^* (dy_2) =dx_2$. We conclude that $D_0 = 0$ near a generic point of $E$, and $\ord_E(D) =\ord_E(\Delta_\f)$.

Suppose next $E$ is $\cG$-invariant but is not contracted. Again we may assume $\f(x_1,x_2) = (x_1^a,x_2)$, but then  $\cF$ is determined by $dy_1$. Now $\f^* (dy_1) = a x_1^{a-1} dx_1$ and $\det D\f = a x_1^{a-1}$ so that $D_0 = (a-1) E = \Delta_\f$ at a generic point of $E$, and $\ord_E(D) =0$.

Finally suppose $E$ is $\cG$-invariant and contracted to a point $p$. Assume first $p$ is smooth on $X$.
Let $\pi : \hat{X} \to X$ be a composition of point blow-ups such that the map $\hat{\f} : Y \dto \hat{X}$
satisfying $ \pi \circ \hat{\f} = \f$ does not contract $E$. Write $\hat{\cF} \= \pi^* \cF$. As $p$ is a reduced singularity, $T^* \hat{\cF} = \pi^* T^* \cF + E'$ with $E'$ an effective divisor supported on the exceptional set of $\pi$. As $E$ is not contracted by $\hat{\f}$ we may apply our previous arguments. We may thus write
$\hat{\f}^* T^* \hat{\cF} -  T^* \cG = - \hat{D}$ for some divisor $\hat{D}$ such that
$\ord_E(\hat{D}) = 0$.
Now we get
\[\f^* T^*\cF -  T^* \cG
=
\hat{\f}^* \pi^*T^*\cF -  T^* \cG
=
\hat{\f}^* T^* \hat{\cF} -  T^* \cG - \hat{\f}^* E'
= - \hat{D} - \hat{\f}^* E' \]
We conclude that $\ord_E(D) = \ord_E(\hat{D})+ \ord_E(\hat{\f}^* E') = \ord_E(\hat{\f}^* E') \ge 0$.

Suppose now $E$ is contracted to   a quotient singularity $p$. Consider $\pi: (\hat{X}, \hat{\cF}) \to (X,\cF)$ a resolution of singularity of $X$, and let $\hat{\phi} : (Y,\cG) \dashrightarrow (\hat{X},\hat{\cF})$ be the lift of $\phi$.
Since $T^*\hat{\cF} = \pi^* T^*\cF + D'$ with $D'$ effective, we have:
$$
\phi^* T^*\cF - T^*\cG = \hat{\phi}^* \pi^* T^*\cF - T^*\cG  = \hat{\phi}^* T^*\hat{\cF} - T^*\cG - \hat{\phi}^* (D')
$$
Since $\hat{\cF}$ has reduced singularities by assumption, we may apply our former computation to $\hat{\phi}$
and we conclude that $\phi^* T^*\cF - T^*\cG = -D$ with $\ord_E(D)\ge 0$.
\end{proof}

\subsection{Proof of Theorem~C}

We actually prove a stronger result, see the proposition below.

As mentioned in the introduction $\deg(\mathcal F)$, the degree of  a foliation $\mathcal
F$ on $\mathbb P^2$, is defined as the number of tangencies between
$\mathcal F$ and a generic line $\ell$. If $\cF$ is defined by a section
with isolated zeroes $\om\in H^0(\mathbb P^2,\Om^1_{\mathbb P^2} \otimes N\cF)$ and
$N\cF = \cO_{\mathbb P^2}(k)$ for some $k$, then the restriction $\om|_\ell$
is a section of $H^0(\mathbb P^1, \Omega^1_{\mathbb P^1} \otimes \cO_{\mathbb P^1}(k) )$.
By definition the number of zeroes of $\om|_\ell$ is equal to $\deg(\cF)$ so that
$$N \mathcal F = \mathcal
O_{\mathbb P^2}(\deg(\mathcal F)+2) \quad \text{ and } \quad
T^*\mathcal F = \mathcal O_{\mathbb P^2}(\deg(\mathcal F)-1).$$
Recall that the algebraic degree of a rational map  $\f: \mathbb P^2 \dto \mathbb
P^2$ is by definition the degree of  $\f^{-1} \ell$ for a generic line $\ell$.
If we apply Proposition \ref{P:var} to reduced foliations on
$\mathbb P^2$ then we obtain the following result.
\begin{prop}\label{C:DJ}
Let $\f: (\mathbb P^2, \mathcal F) \dashrightarrow(\mathbb
P^2,\mathcal F)$ be a dominant rational map of algebraic degree  $d
> 1$. Suppose $\cF$ is reduced or $\f$ does not contract any curve.
Then $\deg(\cF) \le 1$.
\end{prop}
\begin{proof}
One has
$\deg( \f^* T^* \cF) = d \cdot (\deg(\mathcal F) -1 )$.
Our assumption and Proposition~\ref{P:var} implies $D = T^*\cF - \f^* T^*\cF$ is effective. Thus
$ (\deg (\cF) -1) - d \cdot (\deg(\cF) -1)  \ge 0$, i.e. $\deg(\mathcal F) \le 1.$
\end{proof}
Back to the proof of Theorem~C. Suppose $\f$ is holomorphic. It cannot contract curves, hence by the previous proposition, $\deg(\cF) =0$ or $1$.
Having degree $0$ means $\cF$ is given by the fibration $\{x/y =\mathrm{cst}\}$,
which shows we are in case~(1).
When $\cF$ has degree $1$, it is defined by (a) $ d \log\left( \frac{x}{y} \exp\left( \frac{z}{y} \right)\right)$ or (b) $d \log ( x^{\la} y z^{-1 - \la} )$ with $\la \in \C^*$, see~\cite[\S 1.2]{Jo}.
We now note that critical components of $\f$
are necessarily $\cF$-invariant, as follows from Proposition~\ref{P:var}.

If we are in case (a) then there are only two $\f$-invariant algebraic curves, the lines $\{ x y=0 \}$, and their union is totally invariant
by $\f$. Hence $\f$ or its square is equal to $[x^d: y^d : R(x,y,z)]$. A simple computation shows that a holomorphic map of this form cannot
leave invariant a foliation in the type (a) above.

\smallskip

Suppose now we are in case (b).
If   $\la \notin\Q$ then $\cF$ has
only three invariant algebraic curves which  are totally invariant by $\f$. We are thus
in cases~(2) and~(3) of the theorem.
If $\la =p/q$, then the foliation is given by the fibration $\{ \frac{x^py^q}{z^{p+q}} =\mathrm{cst}\}$ which admits a unique reducible component $\{xy=0\}$. This component is necessarily totally invariant, hence
$\f = [x^d:y^d:R]$ or $[y^d:x^d:R]$ with $R$ homogeneous of degree $d$.
In the former case, since $\{xy=0\}$ is a totally invariant fiber, we may find a polynomial $P$ of degree $d$ such that  $R^{p+q}(x,y,1)= (x^py^q)^d\,P(\frac1{x^py^q})$. This implies all irreducible factors of $R(x,y,1)$ are of the form $1+c x^py^q$ with $c\in\mathbb{C}^*$ as required. In the latter case, the equation becomes
$R^{p+q}(x,y,1)= (x^qy^p)^d\,P(\frac1{x^py^q})$. Since $R(x,y,1)$ is a polynomial, this forces $p=q=1$. \qed

\section{Foliated Mori theory}\label{S:mori}

\subsection{Basics}
For the convenience of the reader, we recall the main aspects of foliated Mori theory as developed by Miyaoka, McQuillan~\cite{McQ}, Brunella~\cite{Br,Br2} and Mendes~\cite{mendes}. Its goal is to classify holomorphic foliations on projective surfaces in terms of the positivity properties of their cotangent bundles.

Recall that a class $\alpha \in \NS(X)$ is pseudo-effective if it lies in the closure
of the convex cone generated by effective divisors. It is nef if $\alpha \cdot C \ge0$ for any curve $C$.

The first important result of this theory is due to Miyaoka.
\begin{thm}[{\bf Miyaoka}]
If $(X,\mathcal F)$ is a reduced foliation on a projective surface and $T^* \mathcal F$ is not
pseudo-effective, then $\mathcal F$ is tangent to a rational fibration.
\end{thm}
We may thus turn our attention to reduced foliations $(X,\cF)$ with $T^*\cF$ pseudo-effective.
The next fundamental result is due to McQuillan.
\begin{thm}[{\bf McQuillan}]\label{T:McQ}
Suppose $(X,\mathcal F)$ is a reduced foliation and $T^* \mathcal F$ is pseudo-effective.
Then one can find a regular birational map $\pi: X \to X_0$ to a projective surface $X_0$ with at most cyclic quotient singularities such that $\cF_0 \= \pi_* \cF$ is smooth at any of the singularities of $X_0$ and
 a suitable power of $T^*\cF_0$ is a line bundle which is  \emph{nef}.
\end{thm}
For sake of convenience, we introduce the following terminology.
\begin{definition}
A nef foliation is a foliation $\cF$ on a projective surface $X$
with at most cyclic quotient singularities such that: $\cF$ is smooth at any of the singular points of $X$; it is reduced at all regular points of $X$; and a suitable power of $T^*\cF$ is a nef line bundle.
\end{definition}
The preceding two theorems can be thus rephrased as follows: either $\cF$ is tangent to a rational pencil,
or it is nef in some birational model.

\smallskip

The classification of nef foliations is done according to the values of two invariants called the Kodaira dimension and the numerical Kodaira dimension. The Kodaira dimension of a reduced foliation $(X,\mathcal F)$
is by definition the Kodaira-Iitaka dimension of its cotangent bundle, that is:
\[
  {\rm kod}(\mathcal F)= \limsup_{n \to \infty}
\frac{ \log \mathrm h^0(X, (T^* {\mathcal F})^{\otimes n}) }{\log n} \, .
\]
It is not hard to check that two reduced foliations that are birationally equivalent have the same Kodaira dimension. The numerical Kodaira dimension of a reduced foliation $(X,\mathcal F)$ is defined in terms of
the Zariski decomposition of $T^*\cF$. If $T^* \mathcal F$ is pseudo-effective then we can write $T^*\cF = P_\cF + N_\cF$ in $\NS(X)$ with $P_\cF$ a nef $\Q$-divisor and $N_\cF$ an effective $\Q$-divisor with contractible support such that $P_\cF \cdot N_\cF =0$. Such a decomposition is unique. Notice that $T^* \mathcal F = P_{\mathcal F}$ for nef foliations.  Then we set:\[
\nu(\mathcal F) = \left\{ \begin{array}{rcl}
                     -\infty & \text{ when } & T^* \mathcal F \, \, \text{  is not pseudo-effective};\\
                           0 & \text{ when } & P_{\mathcal F} = 0;  \\
                           1 & \text{ when } & P_{\mathcal F}^2 = 0 \text{ but } P_{\mathcal F} \neq 0;  \\
                           2 & \text{ when } & P_{\mathcal F} ^2>0.
                          \end{array} \, \right.
\]
Again two birationally conjugated reduced foliations have the same numerical Kodaira dimension.
The classification of foliated surfaces is summarized in  the following table:
\vskip0.2cm \centerline{\begin{tabular}{|c|c|p{8cm}|} \hline
\multicolumn{1}{|c|} %
{ { $\nu(\mathcal{F})$} } &{ {\rm{kod}($\mathcal{F}$)} } & Description\\
\hline
\hline
$-\infty$ & $-\infty$ &   Rational fibration \\
\hline
$0$ & $0$ &  $\mathcal{F}$ is the quotient of a foliation generated by a global holomorphic vector field by a finite cyclic group.\\
\hline
$1$ & $-\infty$ &   Hilbert Modular foliation  \\
\hline
$1$    & $1$    &   Riccati foliation      \\
\cline{3-3}
&       & Turbulent foliation \\
\cline{3-3}
 &      & Nonisotrivial elliptic fibration \\
\cline{3-3}
  &     & Isotrivial fibration of genus $\ge$ 2\\
\hline
$2$ & $2$ & General type \\
\hline
\end{tabular}}

\medskip
Recall that a foliation  is a Riccati (resp. turbulent) foliation, if
there exists a fibration $\pi : X \to B$ whose generic fiber is rational (resp. elliptic), and transversal
to $\mathcal F$. A foliation $(X,\mathcal F)$ is a Hilbert modular foliation if, up to birational morphisms, there exists a Zariski open subset $U$ of $X$ which
is isomorphic to the quotient space $\mathbb H^2 / \Gamma $ where $\HH$ is the upper half plane and $\Gamma$ is an irreducible lattice in the product $\mathrm{PSL}(2,\mathbb R)\times \mathrm{PSL}(2,\mathbb R)$  and the restriction of $\mathcal F$  to $U$ is the quotient by $\Gamma$ of one of the two natural fibrations $\mathbb H^2 \to \mathbb H$ .

\subsection{General properties of $\f$-invariant  nef foliations}

\begin{prop}\label{P:key2}
Suppose $(X,\cF)$ is a nef foliation, and $\f$ is a dominant non-invertible rational map preserving $\cF$.
Then, we have:
\begin{equation*}
(T^*\cF)^2 =0 \text{  and  } \f^* T^*\cF = T^* \cF \text{ in } \NS(X)~.
\end{equation*}
Moreover, the  critical set of $\f$ is invariant by $\cF$.
\end{prop}
\begin{proof}
Denote by $e(\f)$ the topological degree of $\f$: by assumption it is an integer greater or equal to $2$.
Consider $\pi: \hat{X} \to X$ the minimal desingularization of $X$.
By Theorem~\ref{T:McQ} the lift $\hat{\cF}$ of $\cF$ to $\hat{X}$ has reduced singularities everywhere. We may thus apply Proposition~\ref{P:var} to the map $\hat{\f}: (X,\cF) \dashrightarrow (\hat{X},\hat{\cF})$ induced by $\f$. Since $\pi_* T^*\hat{\cF} = T^*\cF$, we have the equality  $\f^* T^*\cF = T^*\cF - D$ in $NS(X)$ with $D$ effective.

Because $T^*\cF$ is nef, we get $(T^*\cF)^2 \ge (\f^* T^*\cF)^2$. But the latter term is always greater or equal to $e(\f) \cdot (T^*\cF)^2$, see for instance~\cite[Corollary~3.4]{DF}, hence $(T^*\cF)^2 =0$.
We also get $ 0 \le \f^* T^*\cF \cdot T^*\cF \le  (T^*\cF)^2 =0$, hence $\f^* T^*\cF = \la T^*\cF$ in $\NS(X)$ for some non negative constant $\la$ by Hodge index theorem. Since $\f$ preserves the lattice in $\NS(X)$ generated by classes associated to curves, and since $\f^* T^*\cF \le  T^*\cF$, we have $\la =1$ or $0$. But the relation $\f_* \f^* = e(\f) \mathrm{id}$ implies that $\f^*$ is injective.  We thus have $\f^* T^*\cF = T^* \cF$ in $\NS(X)$ whence $D =0$ numerically. But $D$ is effective according to final remark in Proposition \ref{P:var}, so that $D=0$ as a divisor. Hence the critical set of $\f$ must be $\mathcal F$-invariant. This concludes the proof.
\end{proof}

\begin{prop}\label{P:key3}
Suppose $(X,\cF)$ is a nef foliation, and $\f$ is a dominant non-invertible rational map preserving $\cF$. Denote by $\cC$ the closure of the $\cF$-invariant compact curves, and write $\cU \= X \setminus \cC$.

Then either $\cU$ is empty and $\cF$ is tangent to a fibration; or  $\cU$ is a Zariski-dense open subset of $X$. In the latter case, we may contract finitely many curves on $X$ such that $\cF$ remains a nef foliation and $\f$ induces a holomorphic proper unramified finite covering $\f : \cU \to \cU$. In particular, $X \setminus \cU$ is a $\f$-totally invariant proper analytic subset of $X$ which is maximal for the inclusion.
\end{prop}
Note  that $\mathcal U$ may be equal to $X$. Notice also that  $\cU$ may have quotient singularities. When it is the case $\f$ is unramified in the orbifold sense, which means that the critical locus of $\f$ is empty outside the singular points, and $\f$ has finite fibers and is proper.
For sake of convenience, we introduce the following terminology
\begin{definition}
Suppose $(X,\cF)$ is a nef foliation, and $\f$ is a dominant rational map preserving $\cF$. Then the foliation $\cF$ is \emph{$\f$-prepared} if the complement $\cU$ of all $\cF$-invariant curves is a non-empty dense Zariski open subset of $X$ which is $\f$-invariant, and such that the restriction map $\f: \cU \to \cU$ is an unramified orbifold cover.
\end{definition}
Proposition \ref{P:key3} says that any nef foliation without a rational first integral admits a model in which it is $\f$-prepared.
\begin{proof}
By Jouanolou's theorem, if $\cF$ admits infinitely many invariant compact curves, then it is tangent to a fibration. Thus $\cU$ is either empty or the complement
of finitely many algebraic curves, hence Zariski-dense. For the rest of the proof, assume we are in the latter case.
We let $\Gamma$ be a desingularization of the graph of $\f$, and $\pi: \Gamma \to X$, $f: \Gamma \to X$ be the two natural projections with $\pi$ birational, and $f = \f \circ \pi$.

Proposition~\ref{P:key2} implies  $\f^* T^*\cF = T^* \cF$. Thus there is no critical curve intersecting $\cU$ according to  Proposition~\ref{P:var}~(1).

Suppose $p \in \ind(\f) \cap \cU$. Pick a resolution $\pi: \hat{X} \to X$ of $X$ at $p$ such that
the induced map $\hat{\f}: \hat{X} \to X$ is holomorphic at any point in $\pi^{-1} (p)$, and pick a  neighborhood $V$ around $p$.   Then $f$ induces a proper holomorphic map from the neighborhood $\pi^{-1}(V)$ of the divisor $C\= \pi^{-1}(p)$ in $\Gamma$ to a neighborhood of $f(C)$ that is unramified outside $C$.
When $p$ is smooth, $V\setminus \{ p \}$ is simply connected, hence $f$ is birational, and we may contract $f(C)$ to a smooth point.
When $p$ is a quotient singularity, we can still contract $f(C)$ to a point $q$.
Since $\f$ induces an unramified cover of a complement of a small neighborhood $p$ onto a complement of a small neighborhood of $q$, $q$ is also a  quotient singularity
and $\cF$ is smooth at $q$. But this implies the restriction of $T^*\cF$ to $f(\pi^{-1}(V))$ to have full support on $f(C)$ and to be effective. Since $f(C)$ is contractible, the intersection form on the free abelian group of divisors supported on $f(C)$ is negative definite.
Thus $T^*\cF$ cannot be nef which yields a contradiction.

We conclude that we may contract finitely many curves in $\cC$ such that $\cF$ remains nef and $\f$ becomes holomorphic at any point in $\cU$.

\smallskip

We now show that $\mathcal U$ is totally invariant by $\f$. It is clear that it is forward invariant.
Since $\cF$ has no dicritical singularities, for any point of indeterminacy $p\in X$ the divisor $f(\pi^{-1}(p))$ is $\cF$-invariant, hence included in $\cC$. Suppose by contradiction that there exists a connected curve $C\subset \cC$ contracted by $\f$ to a point $p \in \cU$. By the preceding remark, $C$ does not intersect $\ind(\f)$. Replace $C$ by the connected component of $\f^{-1}(p)$ containing it, and pick a small tubular neighborhood $U$ of $C$. Then $\f$ induces a holomorphic map from $U$ onto its image, whose critical set is included in $C$. We thus get an unramified finite covering from $U \setminus C$ onto the complement of $p$ in a small neighborhood. As before   we can contract $C$ to a smooth point, and after finitely iterations we get a new model in which $\pi( f^{-1}(\cU)) \subset \cU$. The inclusion is in fact an equality since $f$ is surjective from $\Gamma$ onto $X$. We conclude that $\f$ induces a proper finite unramified map from $\cU$ onto itself.
\end{proof}

\section{The classification}\label{S:clas}

\subsection{Kodaira dimension $0$}
If $(X,\mathcal F)$ is a nef foliation of Kodaira dimension zero  then, according to \cite[Theorem IV.3.6]{McQ},
$(X,\mathcal F)$ is the quotient of a  foliation $(Y,\mathcal G)$ with trivial canonical bundle on a smooth surface $Y$ by a finite cyclic group $H$
acting on $Y$ without pseudo reflections. Thus the action has isolated fixed points and the points in $X$ below these fixed
points are singular. We summarize the possibilities in the following table, for the order of the groups see \cite{jvpkod0}.

\smallskip
\begin{center}
{\small
\begin{tabular}{|l|l|c|}
\hline
{\bf Ambient Space} $Y$ & {\bf Foliation} $\mathcal G$& {\bf order of } $H$ \\
\hline
\hline
Sesquielliptic surface   & Isotrivial    elliptic  fibration & 1,2,3,4,6\\
\hline
Abelian Surface  & any vector field not tangent & 1,2,3,4,5,\\ & to an elliptic fibration &6,8,10,12\\
\hline
extension of an elliptic curve by $\mathbb C^*$ & Suspension  of a representation  & 1,2,3,4,6\\
compactified as a $\mathbb P^1$-bundle & $\pi_1(E) \to \mathbb C^*$& \\
\hline
extension of an elliptic curve by $\mathbb C$ & Suspension  of a representation& 1,2\\
compactified as a $\mathbb P^1$-bundle&  $\pi_1(E) \to \mathbb C$ & \\
\hline
$\mathbb C^* \times \mathbb C^*$ compactified & any irrational vector field in the & 1,2,3,4,6 \\
 as $\mathbb P^2$ or $\mathbb P^1\times \mathbb P^1$&  Lie algebra of $\mathbb C^*\times \mathbb C^*$
 &\\
\hline
$\mathbb C^* \times \mathbb C$ compactified &any vector field in the  &  1,2\\
 as $\mathbb P^2$ or $\mathbb P^1\times \mathbb P^1$& Lie algebra of $\mathbb C^*\times \mathbb C$ with  non& \\  &
trivial projections to both factors&\\
\hline
\end{tabular}
}
\end{center}

\medskip
From the classification, and as pointed out in \cite{McQ}, follows the following fact.

\smallskip

\noindent{\bf Fact}: when the original foliation has no rational first integral, the space $Y$ is a compactification of a complex Lie group $G$ with abelian Lie algebra and the foliation $\mathcal G$, when restricted to the Lie group, is induced by a Lie subalgebra. In all these cases $G$ is the quotient of $(\mathbb C^2,+)$ by a subgroup $\Gamma$ of translations, the pullback of $\mathcal G$
to $\mathbb C^2$ is a linear foliation and
$H$ is generated by a cyclic element $h \in \mathrm{GL}(2,\mathbb C)$ with both eigenvalues distinct from one.

\medskip

\begin{lemma}\label{L:lift}
Let $\f:(X, \mathcal F) \dashrightarrow (X,\mathcal F)$ be a rational
self-map of a nef foliation with $\mathrm{kod}(\mathcal F) =0$. Assume $\cF$ is $\f$-prepared. If
we write $(X,\mathcal F)$ as the quotient of $(Y,\mathcal G)$ by a cyclic group $H$ as in the beginning of this section
then the map $\f$ lifts to a rational map $\hat \f :(Y,\mathcal G) \dashrightarrow (Y,\mathcal G)$.
\end{lemma}
\begin{remark}
Below we give a topological proof.  Another proof can be obtained in the spirit of~\cite[p.209-210]{CF}.
\end{remark}
\begin{proof}
Let $\cU$ be the complement of all $\cF$-invariant curves.
Because $\pi :Y \to X$ is a finite map and $\mathcal U$ is Zariski dense, it suffices to lift the restriction of $\f$ to $\mathcal U$. Note that the preimage of $\cU$ by $\pi$ coincides with the Lie group $G$ alluded to above. By assumption, the map $\f|_{\mathcal U}$ is an orbifold covering of $\mathcal U$.
Since the  restriction to $G=\pi^{-1}(\mathcal U)$ of the natural
quotient map $\pi: Y \to X$ is also an orbifold covering , the composition map $g = \f \circ \pi :  G \to \cU$ is an orbifold covering.
It follows from the diagram
\[
 \xymatrix{
 G \ar[d]_{\pi} \ar[dr]^{g} \ar@{-->}[r]^{\exists {\hat{\f}}_{|G} ? }  & G\ar[d]^{\pi}
  \\
 \mathcal U  \ar@{>}[r]^{\f_{|\mathcal U}} & \mathcal U}
\]
that to  lift $\f_{|\mathcal U}$ to $\hat \f: Y \dashrightarrow Y$
it suffices to check that $ g_* \pi_1(G) \subset \pi_* \pi_1(G)$, where $g_*:\pi_1(G) \to \pi_1^{\mathrm{orb}}(\mathcal U)$
is the natural map induced by $g$. For the definition of the orbifold fundamental group $\pi_1^{\mathrm{orb}}$ and its basic properties, the reader can consult \cite{Hae,Th}.

Since $g$ is an orbifold covering , $g_*$ is injective. In particular, every element in $g_* \pi_1(G)$ distinct from the identity has infinite order. To conclude it is sufficient to prove that any element $f\neq \mathrm{id} \in  \pi_1^{\mathrm{orb}}(\mathcal U)$ of infinite order belongs to $\pi_* \pi_1(G)$.
The group $\pi_1^{\mathrm{orb}}(\mathcal U)$
can be interpreted as a subgroup of $\mathrm{Aff}(2,\mathbb C)$.
Write $f(z) = A z + B$  with $A \in \mathrm{GL}(2, \mathbb C)$ and $ B\in \C^2$.
Recall from the fact above that $\pi_1^{\mathrm{orb}}(\mathcal U)$ is an extension of a finite cyclic group generated by some element $h\in\mathrm{GL}(2, \mathbb C)$ which is not a pseudo-reflection by a subgroup of translations canonically isomorphic to $\pi_1(G)$.
If $A \neq \mathrm{id}$, then $A$ is a power of $h$, hence is periodic of period $k>1$. Since $h$ is not a pseudo-reflection, one has
$
f^{\circ k}(z) = A^k z + \left(\sum_{i=0}^{k-1} A^i \right) \cdot B = z \, .
$
This contradicts our assumption. Whence $f$ is a translation, and belongs to $\pi_*\pi_1(G)$. This concludes the proof.
\end{proof}

\begin{thm}\label{T:clas-kod0}
Let $(X,\mathcal F)$ be a foliation on a projective surface with $\kod(\cF) =0$  and without rational first integral. Suppose $\f$ is a dominant non-invertible rational map preserving $\cF$.
Then up to birational conjugacy and to a finite cyclic covering  of order $N$, we are in
one of the following five cases.
\begin{enumerate}
\item[$\kappa_0$(1):] The surface is a torus $X = \C^2/ \Lambda$, $\cF$ is a linear foliation and $\f$ is a linear diagonalizable map.

\noindent In this case $N\in\{1,2,3,4,5,6,8,10,12\}$.
\item[$\kappa_0$(2):] The surface is a ruled surface over an elliptic curve $\pi: X \to E$ whose monodromy
is given by a representation $\rho: \pi_1(E) \to (\C^*,\times)$. In affine coordinates
$x$ on $\C^*$ and $y$ on the universal covering  of  $E$,  $\cF$ is induced by the $1$-form $\omega =dy + \la x^{-1} d x$ where $\lambda \in \mathbb C$, and $\f(x,y) = (x^k,ky)$ where $k\in \Z \setminus \{ -1, 0, 1 \}$.

\noindent In this case $N\in \{1,2,3,4,6\}$.
\item[$\kappa_0$(3):] Same as in case~$\kappa_0$(2) with $\rho: \pi_1(E) \to (\C,+)$, $\omega= dy + \lambda dx$  where $\lambda \in \mathbb C$, $\f(x,y) = (\zeta x+b,\zeta y)$ where
$b \in \mathbb C$, and $\zeta \in \C^*$ with
$|\zeta| \neq 1$.

\noindent In this case $N \in \{1,2\}$.
\item[$\kappa_0$(4):] The surface is $\PP^1 \times \PP^1$, the foliation $\cF$ is given in affine coordinates by the form
$\la x^{-1} dx + dy$ and $\f(x,y) = (x^k, ky)$ with $\la \in \C^*$ and $k \in \Z \setminus \{ -1, 0, 1\}$.

\noindent In this case $N\in \{1,2\}$.
\item[$\kappa_0$(5):] The surface is $\PP^1 \times \PP^1$, the foliation $\cF$ is given in affine coordinates by the form
$\la x^{-1}dx + \mu y^{-1} dy$ and $\f(x,y) = (x^ay^b, x^cy^d)$ where $\la,\mu \in \C^*$,
with $M= \left[\begin{smallmatrix}
a &b \\c &d
\end{smallmatrix}\right]
\in \GL(2,\Z)$, with $\la/\mu \notin \Q_+$, $|ad-bc| \ge 2$, and $M$ diagonalizable over $\C$.

\noindent In this case $N\in \{1,2,3,4,6\}$.
\end{enumerate}
\end{thm}

\begin{proof}Most of the proof follows immediately from the classification stated in the beginning of this section
and Lemma \ref{L:lift}. What is perhaps not completely evident is the assertion that $\f$ is a linear diagonalizable map in cases $\kappa_0(1)$ and $\kappa_0(5)$.
Let us comment on that. Suppose we are in case $\kappa_0(1)$, and $\f$ is not diagonalizable.
Then in $\C^2$ we can write $\f(x,y) = (\zeta x + y , \zeta y)$ for some $\zeta \in \C^*$, and $\f$ preserves a lattice $\Lambda \subset \C^2$. Forgetting about the complex structure, we may now view $\f$ as a linear map on $\R^4$ preserving $\Z^4$ with eigenvalues $\zeta$ and $\bar{\zeta}$ (each of multiplicity $2$). The characteristic polynomial of $\f$ is thus the square of a quadratic polynomial $X^2 - aX +b$ with integral coefficients. From the Jordan decomposition of $\f$, we see that there exists a unique $\f$-invariant $2$ dimensional real plane $P$. In $\R^4$, it is given by $P= \mathrm{ker}\, (\f^2 -a\f+b\mathrm{id})$, and is hence defined over $\Z$. In $\C^2$, it equals $P = \{ y =0 \}$ and is invariant under complex conjugation.
We have thus proved that the intersection of $\Lambda$ with the complex line $\{ y =0 \}$ is a rank $2$ lattice. It follows that the foliation $dy$ has a compact leaf. Since translations act transitively on
$\C^2/\Lambda$ by preserving the foliation, we see that all leaves of the foliations are elliptic curves which implies the foliation to have a rational first integral. A contradiction.

In the case $\kappa_0(5)$, the matrix is again diagonalizable. Otherwise, one has
$\f(x,y) = (x^d, xy^d)$ with $d \in \Z$, and the invariant foliation is the rational fibration given by $\{ x= \mathrm{cst}\}$.
\end{proof}

\medskip

\subsection{Kodaira dimension $1$} The next result classifies the pairs $(\cF, \f)$ when $\cF$ has Kodaira dimension one and
does not admit a rational first integral.

\begin{thm}\label{T:clas-kod1}
Let $(X,\mathcal F)$ be a foliation on a projective surface with $\kod(\cF) =1$  and without rational first integral.
Suppose $\f$ is a dominant non-invertible rational map preserving $\cF$.
Then up to birational conjugacy and to a finite cyclic covering generated by an automorphism $\tau$ of order $N$, we are in one of the following two cases.
\begin{enumerate}
 \item[$\kappa_1$(1):] The surface $X$ is $\PP^1 \times \PP^1$, the foliation is a Riccati foliation given
       in coordinates by the form
\[
\omega=\frac{dy}{y} + \frac{m\, dx}{(k-1)\,x} +  \mu x^n dx~,
\] and $  \f(x,y)=(\la x, x^m y^k)$,
where $\la, \mu \in \C^*$, $k \in \Z \setminus \{ -1,0,1\}$, $m \in \Z$, $n \in \Z\setminus \{ -1 \}$,
such that $\la$ is not a root of unity and $k = \la^{n+1}$. Moreover $ n \notin \{-2, -1, 0\}$ if $m =0$.

\noindent In this case, $N\in\{1,2\}$, and $\tau= (\pm x^{\pm1}, \pm y^{\pm1})$.
 \item[$\kappa_1$(2):] The surface $X$ is $\PP^1 \times E$ where
 $E$ is an elliptic curve, the foliation is a turbulent foliation induced by the form
\[
\omega= dy + x^n dx~,
\]
where $x$ is a coordinate on $\PP^1$ and $y$ on the universal covering  of $E$,
and $\f(x,y)=(\la x, \la^{n+1}y)$ with $\la \in \C^*$ not a root of unity, $n \in \Z \setminus \{ -2, -1, 0\}$.

\noindent In this case, $\tau= (\zeta x, \xi y)$, $N =1$ when $n\ge1$, and $N\in \{ 1,2,3,4,6\}$ otherwise.
\end{enumerate}
\end{thm}
\begin{proof}
As before, we may suppose $(X,\cF)$ is a nef foliation such that
the complement of the compact $\cF$-invariant curves is Zariski dense and totally invariant by $\f$.
Under the additional condition that $\kod(\cF) =1$, we shall need the following result.
\begin{lemma}\label{l:ricandturb}
There exists a rational or elliptic fibration $\pi: X \to {\mathbb P}^1$
whose generic fiber is transversal to $\cF$ and that admits one or
two $\cF$-invariant fibers.
Moreover, one can find $f\in \aut({\mathbb P}^1)$ which is not periodic,
such that $\pi \circ \f = f \circ \pi$. In particular, the $\cF$-invariant fibers
are totally invariant by $\f$.
Finally singularities of the ambient space $X$ are all located on $\cF$-invariant fibers.
\end{lemma}
A proof is given at the end of this section.

\smallskip

\noindent {\bf 1.} Suppose first that $\pi$ is a rational fibration. Then we may make a sequence of modifications centered on one of the  $\cF$-invariant fibers such that $X$ becomes $\PP^1 \times \PP^1$. Note that the
condition on $\cF$ being nef may  not be satisfied anymore, but  we can keep the properties listed in Lemma~\ref{l:ricandturb}. Let us consider the component $\Delta_\perp$ of the critical set  of $\f$ which is not invariant by the fibration.
A non-invertible rational map of $\PP^1$ has at least two critical points, hence $\Delta_\perp \cdot \pi^{-1}(b) \ge 2$. On the other hand, by Propositions~\ref{P:var} and~\ref{P:key3}, $\Delta_\perp$ is an $\cF$-invariant compact curve that is totally invariant by $\f$, so
that $\Delta_\perp \cdot \pi^{-1}(b)=2$. If $\Delta_\perp$ is irreducible, the fiber product $\Delta_\perp \times_{\pi} X$ is a two-fold covering of $X$ ramified over $\Delta_\perp$. Replacing $X$ by
this fiber product if necessary, we may  assume that $\Delta_\perp$ has
two irreducible components $C_1$ and $C_2$.

We now pick coordinates $(x,y)\in \PP^1 \times \PP^1$ such that $\pi(x,y) = x$, $C_1 = \{ y=0 \}$, $C_2 = \{ y = \infty\}$. We assume $\pi^{-1}(\infty)$ is an $\cF$-invariant fiber, and the other $\cF$-invariant (if it exists) is $\pi^{-1}(0)$.
The foliation $\cF$ can thus be defined by a form $\om = y^{-1} dy + h(x) dx$ for some rational function $h$. Poles of $h$ correspond to $\cF$-invariant fibers, hence we may write $h(x) = x^n p(x)$
for some $n \in \Z$, and some polynomial $p \in \C[x]$ with $p(0) \neq 0$.
We also have  $\f(x,y) = (f(x), a(x) y^k)$
with $a(x) \in \C(x)$ and $f(x) = x+1$ or $= \la x$ with
$\la \in \C^*$ not a root of unity.

Assume first that $f(x) = \la x$. Since $\cF$ is $\f$-invariant, we have $\f^* \om= b \cdot \om$ for some meromorphic function $b$. Expanding this equation in terms of $x$ and $y$, and identifying both hand sides, we find that
$x^n (k p(x) - \la^{n+1} p(\la x)) = a'(x)/a(x)$.
The left-hand side is holomorphic outside $0 \in \mathbb C$, thus $a(x) =  a_0 x^m$ for some $m \in \Z$. By changing the coordinate $y$ by $a_0^{1/k-1}y$, we may assume $a_0 =1$.
Therefore, we obtain
\begin{equation*}
x^{n+1} \left( k p(x) - \lambda^{n+1} p(\lambda x)\right) = m \, .
\end{equation*}
Elementary considerations now show that we necessarily have
$$\phi(x,y)   = (\la x, x^m y^k), \text{ and } \om = \frac{dy}{y}+  \frac{m\, dx}{(k-1)x}  + \mu x^n dx~.$$
with $\la \in \C^*$, $\mu\in\C$, $k \ge 2$, $m \in \Z$, $n\in \Z\setminus\{ -1\}$, such that $\la$ is not a root of unity and $k = \la^{n+1}$.

Since $\cF$ has no rational first integral, $\mu$ is non zero. One can then compute $T^*\cF$ by relating it to $\pi^* K_{\PP^1}$, see for instance~\cite[\S 4]{Br}:
\begin{eqnarray}
&& T^*\cF = \pi^* \cO_{\PP^1}(-2) + \max \{ 1, -n \} \pi^{-1}(0) + \max \{ 1, n+2 \} \pi^{-1}(\infty)
\text{ if } m \neq 0, \label{e:can1}  \nonumber \\
&& T^*\cF = \pi^* \cO_{\PP^1}(-2) + \max \{ 0, -n \} \pi^{-1}(0) + \max \{ 0, n+2 \} \pi^{-1}(\infty)
\text{ if } m =0~. \label{e:can2}
\end{eqnarray}
Whence $\kod(\cF) = 1$ if and only if  $m\neq 0$ and  $n \neq -1$; or $m=0$ and $n \neq -2, -1, 0$ as required.

Assume now that $f(x) = x+1$. The fiber $\pi^{-1}(0)$ being not $\f$-invariant, it is not $\cF$-invariant, and thus $h(x)$ is a polynomial.
Similarly to the previous case, we need to solve $k h(x) - h(x+1) = a'(x)/a(x)$. Because the left-hand side has no poles $a(x)$ must be constant, and consequently  $h$ must satisfies the functional equation
$kh(x) - h(x+1) = 0$, which is impossible. This concludes the proof of the theorem in the case $\pi$ is a rational fibration.

\smallskip

\noindent {\bf 2.} Suppose now that $\pi$ is an elliptic fibration.
By Lemma~\ref{l:ricandturb}, the action of $\f$ on the base of the fibration is not periodic, hence the fibration is isotrivial. As in the previous case, pick coordinates $x\in \PP^1$ such that the fiber $\pi^{-1}(\infty)$ is  $\cF$-invariant, and the other possible $\cF$-invariant fiber is $\pi^{-1}(0)$.
Then $f(x) = \la x$ with $\la \in \C^*$ not a root of unity, or $f(x) = x+1$. We also pick an affine coordinate $y$ on the universal covering  of  $E$.

Suppose first $f(x) = \la x$.
After a base change of the form $x \mapsto x^l$,  one can assume that $X = \mathbb P^1 \times E$ for some elliptic curve $E$. We may thus write $\f(x,y)=( \lambda x , \varphi(y))$ for some non-invertible $\varphi \in \mathrm{End}(E)$, and the foliation $\cF$ is induced by a form $\omega = dy + h(x) dx$ with $h \in \C(x)$ with poles only at $0$ and $\infty$. Denote by $\varphi_0$ the multiplier of $\varphi$.
It is an integer if $E$ has no complex multiplication, and a quadratic integer otherwise.
Since $|\varphi_0|^2$ is the topological degree of $f$, we have $|\varphi_0|>1$.
The $\f$-invariance of $\cF$ reads in this case as $\lambda h(\lambda x) = \varphi_0\cdot h(x)$.
Thus $h(x)=\mu x^n$, $\varphi_0 = \lambda^{n+1}$ with $n \in \Z \setminus\{-1\}$ and $\mu \in \C^*$.
Replacing $x$ by $\mu^{1/n+1}x$, we may take $\mu =1$.
This shows $\f = (\la x, \la^{n+1}y)$ and $\om = dy + x^n dx$.
Finally $T^*\cF$ can be computed as before. One finds that it satisfies~\eqref{e:can2}.
Hence $\kod(\cF) = 1$ if and only if  $n \neq -2, -1, 0 $ as required.

Suppose now $f(x) = x+1$. Then there is only one $\cF$-invariant fiber, namely $\pi^{-1}(\infty)$. Since $\C$ is simply connected, the monodromy of the fibration
 is trivial, and we may assume $X = \PP^1 \times E$, $\f(x,y)=( x+1, \varphi(y))$
and $\cF$ is induced by a one-form $\omega = dy + h(x) dx$
as before. The $\f$-invariance of $\cF$ implies $h(x+1) = \varphi_0\cdot h(x)$. But this equation has no solution when $|\varphi_0| >1$. This concludes the proof.
\end{proof}

\begin{proof}[Proof of Lemma~\ref{l:ricandturb}]
Since $\kod(\cF) =1$ and $\cF$ does not admit  a rational first integral, for  suitably large $n$,  the natural map $X \to \mathbb{P}(H^0(T^* \cF^{\otimes n}))^*$ induced by the complete linear series induces a fibration $\pi: X \to B$
whose generic fiber is either rational or elliptic. By construction, the class of a generic fiber of $\pi$ is a multiple of  $(T^*\cF)$. The foliation is moreover transversal to a generic fiber of $\pi$, see~\cite[\S 9.2]{Br}.
Since $\f^* T^*\cF = T^*\cF$ in $\NS(X)$ by Proposition~\ref{P:key2}, the map $\f$ preserves this fibration and induces a holomorphic map $f$ on the base that is invertible.

Suppose first that $f = \mathrm{id}$ in $B$. Then $\f$ induces a holomorphic map on a generic fiber whose topological degree is $e(\f) \ge 2$. In particular, replacing $\f$ by a suitable iterate, we may assume that it admits at least three fixed points that are repelling along the fiber. Let $C$ be the irreducible component of the fixed point set of $\f$ passing through such a point $p$. Then we have $\pi(C) = B$, and $C \cdot \pi^{-1}(b) \ge 3$ for any $b \in B$.  At any point near $p$, the differential of $\f$ has two eigenvectors, one with eigenvalue $1$ and tangent to $C$, and the other one with eigenvalue of modulus $>1$ and tangent to the fiber. Since $\f$ preserves $\cF$ and the foliation is transversal to $\pi$, we conclude that $C$ is an $\cF$-invariant curve.
By Proposition~\ref{P:key3}, the set of $\cF$-invariant compact curves is totally invariant by $\f$.
But no non-invertible holomorphic self-map on a rational or an elliptic curve admits  a totally invariant finite set of cardinality greater or equal to $3$.
We conclude that $f\neq \mathrm{id}$. The same argument shows that $f$ is not periodic, and thus the base $B$ is either elliptic or rational.

Note that we may push this argument further. Indeed, if $f$ admits no totally invariant finite subsets,
then the fibration has no $\cF$-invariant fiber. In this case, $\cF$ is a suspension and has Kodaira dimension $0$.
Therefore $B$ is a rational curve, and $f$ has either one or two totally invariant points, each one of them corresponding to an $\cF$-invariant fiber of the fibration.

To conclude the proof, we need to show that the singular set of $X$ is included in the $\cF$-invariant fibers. To do so, pick an arbitrary point $p \in X$ not lying on an $\cF$-invariant fiber. A local neighborhood $U$ of $p$ is given by the quotient of $\C^2$ by a cyclic group generated by a map $h(x,y)= (\zeta x, \xi y)$ for some roots of unity $\zeta$ and $\xi$. Denote by $g$ the natural projection $g: \C^2 \to U$. The foliation $g^*\cF$ is smooth since $\cF$ is a nef foliation. The fibration
induced by $ \pi \circ g$ in $\C^2$ has to be transversal to $g^*\cF$ otherwise $\cF$ and the fibration would have some tangencies near $p$. Using the invariance by $h$, we conclude that $\pi \circ g (x,y) = x$, and $ \cF$ is given by $dy$. Now at a generic point of the fiber containing $p$, the differential of the map $\pi$ has rank $1$, because $f$ is invertible. This forces $\zeta =1$, whence $X$ is smooth at $p$.
\end{proof}

\subsection{The case of a fibration} We now deal with foliations admitting rational first integrals.

\begin{thm}\label{T:clas-fibr}
Let  $\f$ be a dominant non-invertible rational map preserving a fibration $\pi: X \to B$.
Up to birational conjugacy and to a finite cyclic cover, we are in one of the following four
mutually exclusive cases.
\begin{enumerate}
\item[(Fib1)] The surface $X = B \times F$ is a product of two Riemann surfaces with $g(F) \ge 1$ and $g(B) \le 1$,
the fibration is the projection onto the first factor
$B \times F \to B$, and $\f = (f(x), \varphi(y))$ is a product map with $f$ not periodic.
\item[(Fib2)] The surface is the torus $X = E \times E$ for some elliptic curve, and $\f$ is
skew product $\f = (f(x), \varphi(x)(y))$ with $f$ not periodic, and $\varphi: E \to \mathrm{End}(E)$ is holomorphic and non constant.
\item[(Fib3)] The fibration is elliptic and  the action of $\f$ is periodic on the base.  In other words, some iterate of $\f$ is an  endomorphism of an elliptic curve over the function field $\C(B)$.
\item[(Fib4)] The fibration is rational, $ \pi: X \= \PP^1 \times B \to B$,
and $\f$ is a skew product $\f (x,y) = (f(x), \varphi(x)(y))$
where $f$ is a holomorphic map on $B$, and $\varphi$ is rational map from $B$ to the space of rational maps of a fixed degree on $\PP^1$.
\end{enumerate}
\end{thm}
\begin{proof}
Let $f$ be the induced map on the base.
We may assume $\pi$ is not a rational fibration, otherwise we are in case~(Fib4).
If $f$ is periodic, since $\f$ is not invertible the generic fiber is elliptic and we are in case~(Fib3).
If $f$ is not periodic, the fibration is isotrivial: there exists a Riemann surface $F$ such that
$\pi^{-1}(b)$ is isomorphic to $F$ for almost all $b$.
By the semi-stable reduction theorem, see~\cite[III.10]{BPV}, up to birational conjugacy,
a local model near a singular fiber is given by a quotient $ \mathbb{D}\times F$ by a cyclic group
generated by a map of the form $(x,y) \mapsto (\zeta x, h(y))$ with $\zeta$ a root of unity, and $h$ a finite order automorphism of $F$. To any such singular fiber $\pi^{-1}(b)$ with $b \in B$, we let $n(b) \in \N^*$ be the order of its associated group. The divisors $D_B = \sum (n(b)-1) [b]$ and $D = \sum (n(b)-1) \pi^{-1}(b)$ induce natural orbifold structures on $B$ and $X$ respectively and we denote by $B^{\mathrm{orb}}$, $X^{\mathrm{orb}}$ the associated orbifolds. Then one can rephrase the semi-stable reduction theorem by saying that $\pi: X^{\mathrm{orb}} \to B^{\mathrm{orb}}$ is a locally trivial fibration in the orbifold sense. Any rational map preserving such a fibration is holomorphic, see for instance~\cite{cantat}. Since $f : B^{\mathrm{orb}} \to B^{\mathrm{orb}}$ is holomorphic and not periodic, we are in one of the following situations: $(B,D_B) = (\PP^1 ,0)$, $(\PP^1 , p)$, $(\PP^1, p+q)$,   $(\PP^1,2p)$, or $(E,0)$ for some elliptic curve (here $p \neq q$). For monodromy reasons, the cases $(\PP^1 ,p)$ and $(\PP^1 ,2p)$ are excluded.
A base change of order $2$ allows one to reduce the case $(\PP^1, p+q)$ to $(\PP^1 ,0)$.
In the latter case, we have $X = \PP^1 \times B$. Since by assumption $g(B) \ge 1$,
the map $\f$ is a product: $\f(x,y) = ( f(x), \varphi(y))$. In the case $B$ is an elliptic curve,
since we assume $X$ to be k\"ahlerian, it is a torus up to a finite cover. By Poincar\'e irreducibility theorem, $X$ is a product of two elliptic curves $B \times B'$. When the curves are not isogeneous, we fall into case~(Fib1). Otherwise, we fall into either case~(Fib1) or~(Fib2) of the theorem.
This concludes the proof.\end{proof}

\subsection{Modular foliations}
Let us briefly recall the geometric context in which modular foliations arise. We refer to~\cite{LG} for a detailed description of the construction.
Let $\Gamma$ be an irreducible lattice in $\PSL(2,\R)\times \PSL(2,\R)$. The quotient  $X = \HH^2 /\Gamma$ is a quasi-projective
variety which may be projective or not. In any case, it admits a projective compactification $\bar{X}$  by adding finitely many cusps.
When $X$ is not projective $\hat{X}$ will stand for  the minimal desingularization of  the cusps of $\bar{X}$. To keep the notation uniform
$\hat X$ will be set equal to $X$ when $X$ itself is already projective.
We point out that $\hat X$ is not necessarily smooth, since we are resolving only the cusps, but has at worst  cyclic quotient singularities.

The two natural projections $\HH^2$ induce two foliations by holomorphic disks on the bidisk.
We denote by $\cF$ and $\cG$ their respective images in $X$, $\bar{X}$ and $\hat{X}$.
In $\hat{X}$, $T^*\cF$ and $T^*\cG$ are nef and satisfy $T^*\cF \otimes T^*\cG = K_{\hat{X}} \otimes \cO_{\hat{X}}(D)$ with $D$ supported on the exceptional divisor of the projection $\pi: \hat{X} \to \bar{X}$.
Although the cotangent sheaves of $\mathcal F$ and $\mathcal G$ in $\bar X$ are not locally free, we can still write
 $T^*\cF + T^*\cG = K_{\bar{X}}$ in $\NS(\bar{X})$ if we interpret $T^* \cF, T^* \cG \in \NS(\bar{X})$ as the direct images under $\pi$  of
 $T^* \cF, T^* \cG \in \NS(\hat{X})$.  It is a fact that the two classes
$T^*\cF$ and $T^*\cG$ are not proportional, see for instance~\cite[\S 2.2.4]{LG}.

\begin{prop}\label{P:modular}
Suppose $\cF$ is a modular foliation.
Then any rational map $\f$ preserving $\cF$ is invertible.
\end{prop}
\begin{proof}

We give two independent arguments.

\noindent {\bf First argument.}
We may assume $\f$ is a map on $\hat{X}$ that preserves $\cF$ in the notation above. For sake of convenience, we shall also write $\f$ for its induced map on $\bar{X}$.
Since $T^*\cF$ is nef, we may apply Proposition~\ref{P:key3}. Here $\cU$ coincides with $\mu^{-1}(X)$, hence $\f$ is an unramified holomorphic finite covering  from $X$ to itself. Since $\bar{X}$ is normal and $\bar{X}\setminus X$ is a finite set, $\f$ is holomorphic on $\bar{X}$. Note that the critical set of $\f$ on $\bar{X}$ remains empty. We conclude that in the space $\NS(\bar{X})$, we have $\f^* K_{\bar{X}}= K_{\bar{X}}$ and $\f^* T^*\cF=T^* \cF$, which implies $\f^* T^* \cG = T^* \cG$.
Now the two-dimensional vector space generated by $T^* \cF$ and $T^* \cG$ is $\f^*$-invariant and contains
a nef class $\om$ of positive self-intersection.
Since $\f^*$ is holomorphic, we get $e(\f) \om^2 = (\f^* \om)^2 = \om^2$, hence $\f$ is an automorphism.

\noindent {\bf Second argument.}
By Proposition~\ref{P:key3}, $\f$ is an unramified cover of  $X$, hence preserves $S=\mathrm{sing}(X)$, the set of singular points of $X$.
We deduce that $\f$ induces a non-ramified finite covering  from $X\setminus \f^{-1}(S)$ onto $X\setminus S$.
It can thus be lifted to a map $\hat{\f}$ between the universal covering s of these spaces which are both isomorphic to $\HH^2$ minus a discrete  set. By Hartog's theorem, $\hat{\f}$ extends as a map from $\HH^2$ to itself without
ramification points. In explicit coordinates $(x,y) \in \mathbb H^2$ where the lifting of $\mathcal F$ is given by the projection onto the
first factor we can thus write
\[
\widehat \f(x,y) = \left( \frac{a_1 x + b_1}{c_1 x + d_1 }   ,  \frac{a_2( x) y  + b_2(x)}{c_2(x) y  + d_2(x) } \right)
\]
where $a_1,b_1,c_1,d_1$ are constants, $a_2,b_2,c_2,d_2$ are holomorphic functions in $\mathbb H$, both satisfying $a_i d_i - b_i c_i = 1$.
Because $\widehat \f$ descends to $\mathbb H^2 / \Gamma$ it must satisfy $\widehat \f \, \Gamma \, \widehat{\f}^{-1} \subset \Gamma$.
We point out  that $\Gamma$ is a Zariski dense subgroup of $G$ by Borel density theorem.
Taking Zariski closures in $G$
we see that $\widehat{\f} \, G \,  \widehat{\f}^{-1} \subset G $. The specialization of this last equation
to certain elements of $G$ allows one to conclude that the functions $a_2,b_2,c_2$ and $d_2$ are indeed constants.
Thus  $\widehat{\f} \in G$  and $\widehat{\f} \, \Gamma \, \widehat{\f}^{-1} = \Gamma$. Consequently $\f_{X}$ is a biholomorphism (indeed an
isometry with respect to the products of  Poincar\'{e} metrics).
\end{proof}

\subsection{Proof of Theorem~A and Corollary~B}

Let us begin with Theorem~A.
Pick $\cF$ a foliation without rational first integral. By Seidenberg's,
Miyaoka's and McQuillan's theorems, we may assume $X$ is a model (with at most quotient singularities) in which $\cF$ has reduced singularities and $T^*\cF$ is nef.
By Proposition~\ref{P:key2}, $(T^*\cF)^2 =0$ hence $\cF$ can not be of general type.
Thanks to the classification of foliated surfaces, we are in one of the following three cases:
either $\cF$ is a modular foliation, or $\kod(\cF) =0$, or $\kod (\cF) =1$.
The modular case is excluded by Proposition \ref{P:modular}.

When $\kod(\cF) =0$, we may apply Theorem~\ref{T:clas-kod0}. It is clear that Theorem~A holds in this case: the lift of $\cF$ to $\C^2$ is a foliation induced by a constant vector field. When $\kod(\cF)=1$, we apply Theorem~\ref{T:clas-kod1}. When $\cF$ is a Riccati foliation,
then if we consider the  map $\varphi: \C^2 \rightarrow \C^* \times \C^*$ given by
\[(x,y) \mapsto \left( \exp\left(\frac{x}{n+1}\right), \exp \left( y - \frac{mx}{(n+1)(k-1)}\right) \right)\]
then
\[
 \varphi^* \left( \frac{dy}{y} + \frac{m}{k-1} \frac{dx}{x} +  \mu x^n dx \right)
= dy + \frac{\mu e^x}{n+1} dx \, ,
\]
which is induced by a vector field. On the other hand, $\varphi^{-1} \circ \f \circ \varphi = ( x+\mu, ky + l x + \nu)$  for some $\mu, \nu$ and $l$ and is thus Latt\`es-like.
When $\cF$ is a turbulent foliation, one considers the map
$\varphi: \C \times E \to \C^* \times E$ given by $\varphi(x,y)\= (\exp\frac{x}{n+1}, y)$. Then
$\varphi^* \om = dy +\frac{e^x}{n+1} dx$ and $\varphi^{-1} \circ \f \circ \varphi = ( x+\mu, e^\mu y)$ with
$e^\mu = \la^{n+1}$.
This concludes the proof of Theorem~A.

\begin{proof}[Proof of Corollary~B]
Let $ \f : X \dashrightarrow X$ be a dominant rational map preserving a foliation $\cF$.
Suppose $\kod(\cF)=0$ and $\cF$ has no rational first integral. Then there exists a
projective surface $\hat{X}$, a dominant map $\hat{f}: \hat{X} \dto \hat{X}$, a foliation $\hat{\cF}$, and a cyclic group of automorphisms $G \subset \mathrm{Aut}\,(\hat{X})$ such that $(\hat{\f}, \hat{X}, \hat{\cF})$ is in the list given by Theorem~\ref{T:clas-kod0}, $G$ preserves $\hat{\cF}$ and commutes with $\hat{\f}$ and
 $(\f,X,\cF)$ is birationally conjugated to the quotient
situation on $\hat{X}/G$.
In cases $\kappa_0(2)$, $\kappa_0(3)$ and $\kappa_0(4)$,  the map $\hat{\f}$ preserves a unique rational fibration. This fibration is automatically $G$-invariant, hence $\f$ preserves a rational pencil.
In the case $\kappa_0(1)$, the lift $M$ of $\hat{\f}$ to $\C^2$ is diagonalizable. Since $MG = GM$ and $G$ is cyclic, $M$ and $G$ have two fixed points in common on $\PP(\C^2)$.
The projective space $\PP(\C^2)$ naturally parameterizes the set of all linear foliations, hence $\f$ preserves the two associated foliations.
In the case $\kappa_0(5)$, the same argument applies.

When $\kod(\cF) =1$ and $\cF$ has no rational first integral, we apply Theorem~\ref{T:clas-kod1}. The arguments are essentially the same. In the case $\kappa_1(1)$, the group $G$ that may appear are explicit, and the rational fibrations given by $dx$ and $y^{-1}dy + m[(k-1)x]^{-1}dx$ are both $G$-invariant. In the case $\kappa_1(2)$ there is a unique invariant elliptic fibration.

Finally when $\cF$ is a fibration which is neither rational nor elliptic,  we fall into case (Fib$1$).
\end{proof}

\end{document}